\documentclass[10pt,a4paper]{article}
\usepackage[utf8]{inputenc}
\usepackage[english]{babel}
\usepackage{hyperref}
\usepackage{amsmath,amsthm,enumerate}
\usepackage{amssymb}
\usepackage{tikz}
\usepackage{subfigure}
\usepackage[hmargin=2.5cm,vmargin=3cm]{geometry}
\usepackage[color=green!30]{todonotes}

\usepackage{perpage} 
\MakePerPage{footnote} 

\newtheorem{theorem}{Theorem}

\newtheorem{proposition}[theorem]{Proposition}

\newtheorem{corollary}[theorem]{Corollary}

\newtheorem{definition}[theorem]{Definition}

\newtheorem{claim}{Claim}[theorem]

\newcommand{\LD}{\gamma_L}
\DeclareMathOperator{\SEP}{loc}

\newcommand{\claimproof}{\noindent\emph{Proof of claim.} }
\newcommand{\smallqed}{{\tiny ($\Box$)}}

\begin{document}

\title{Domination and location in twin-free digraphs}
\author{Florent Foucaud\footnote{\noindent Univ. Bordeaux, Bordeaux INP, CNRS, LaBRI, UMR5800, F-33400 Talence, France. E-mail: florent.foucaud@gmail.com}
  \and Shahrzad Heydarshahi\footnote{\noindent Univ. Orl\'eans, INSA Centre Val de Loire, LIFO EA 4022, F-45067 Orl\'eans Cedex 2, France. E-mail: shahrzad.heydarshahi@univ-orleans.fr}
\and Aline Parreau\footnote{\noindent Univ Lyon, Universit\'e Claude Bernard, CNRS, LIRIS - UMR 5205, F69622, France. E-mail: aline.parreau@univ-lyon1.fr}
}

\maketitle

\begin{abstract}
  A dominating set $D$ in a digraph is a set of vertices such that every vertex is either in $D$ or has an in-neighbour in $D$. A dominating set $D$ of a digraph is locating-dominating if every vertex not in $D$ has a unique set of in-neighbours within $D$. The location-domination number $\LD(G)$ of a digraph $G$ is the smallest size of a locating-dominating set of $G$. We investigate upper bounds on $\LD(G)$ in terms of the order of $G$. We characterize those digraphs with location-domination number equal to the order or the order minus one. Such digraphs always have many twins: vertices with the same (open or closed) in-neighbourhoods. Thus, we investigate the value of $\LD(G)$ in the absence of twins and give a general method for constructing small locating-dominating sets by the means of special dominating sets. In this way, we show that for every twin-free digraph $G$ of order $n$, $\LD(G)\leq\frac{4n}{5}+1$ holds, and there exist twin-free digraphs $G$ with $\LD(G)=\frac{2(n-2)}{3}$. Improved bounds are proved for certain special cases. In particular, if $G$ is twin-free and a tournament, or twin-free and acyclic, we prove $\LD(G)\leq\lceil\frac{n}{2}\rceil$, which is tight in both cases.
\end{abstract}

\section{Introduction}

A \emph{dominating set} in a digraph $G$ is a set $D$ of vertices of $G$ such every vertex not in $D$ has an in-neighbour in $D$. The \emph{domination number} $\gamma(G)$ of $G$ is the smallest size of a dominating set of $G$. The area of domination is one of the main topics in graph theory: see the two classic books~\cite{book,book2} on the subject. While there are hundreds of papers on domination in undirected graphs,\footnote{A set of vertices of an undirected graph $G$ is a dominating set if it is a dominating set of the digraph obtained from $G$ by replacing each edge by two symmetric arcs.} domination in digraphs is less studied; we refer to the papers~\cite{fu,lee,meggidovishkin} for some examples. One particular variation of domination is the concept of location-domination, introduced by Slater for undirected graphs in~\cite{slater} (see also~\cite{rs,s2}). For a set $S$ of vertices of a digraph $G$, two vertices $x$ and $y$ of $V(G)\setminus S$ are {\em located} by $S$ if there is a vertex of $S$ that is an in-neighbour of exactly one vertex among $x$ and $y$. The set $S$ is a {\em locating set} of $G$ if it locates all the pairs of $V(G)\setminus S$ (but does not necessarily dominate the graph). Equivalently, every vertex not in $S$ has a distinct set of in-neighbours in $S$. The {\em location number} of a graph $G$ is the size of a smallest locating set and is denoted by $\SEP(G)$. A set $D$ of vertices of a digraph $G$ is \emph{locating-dominating} if it is both locating and dominating. The \emph{location-domination number} $\LD(G)$ of a digraph $G$ is the smallest size of a locating-dominating set of $G$. Note that $\LD(G)-1\leq \SEP(G)\leq \LD(G)$, since at most one vertex is not dominated in a locating set.

Our goal is to investigate bounds on the location-domination number of digraphs. Such bounds are absent from the literature; in fact, the only paper on location-domination in digraphs we are aware of is~\cite{CHL02}, which deals with the computational complexity of the problem. Such bounds on digraphs have been studied for the closely related concept of identifying codes in~\cite{IDcodes}.

We now introduce our terminology. We will assume that all the considered digraphs are loopless and have no multiple arcs. A digraph $G$ contains a set $V(G)$ of vertices and a set $A(G)$ of arcs, that are ordered pairs of vertices. If there is an arc from $v$ to $w$, we say that $v$ is an \emph{in-neighbour} of $w$, and that $w$ is an \emph{out-neighbour} of $v$. The \emph{open in-neighbourhood} and the \emph{open out-neighbourhood} of a vertex $v$, denoted by $N^{-}(v)$ and $N^{+}(v)$, respectively, are the set of in-neighbours of $v$ and the set of out-neighbours of $v$, respectively. Further, the \emph{closed in-neigbourhood} of $v$ is $N^{-}[v] = N^{-}(v) \cup \lbrace v \rbrace$ and the \emph{closed out-neigbourhood} of $v$ is $N^{+}[v] = N^{+}(v) \cup \lbrace v \rbrace$. A \emph{source} is a vertex with no in-neighbours, and a \emph{sink} is a vertex with no out-neighbours. Two vertices are called \emph{twins} if they have the same open in-neighbourhood or the same closed in-neighbourhood. If two vertices $x$ and $y$ satisfy $N^-(x)=N^-(y)\cup\{y\}$, they are called \emph{quasi-twins}. See Figure~\ref{fig:twins} for an illustration. A \emph{directed path} is a sequence of vertices where each vertex has an arc to the next vertex in the sequence. A \emph{directed cycle} is a directed path where the first and the last vertices are the same. An \emph{acyclic digraph} is a digraph with no directed cyle. A \emph{tournament} is a digraph in which there is a unique arc between any pair of vertices. A tournament is \emph{transitive} if it is acyclic. A digraph is called \emph{strongly connected} if there exists a directed path in both directions between every pair of vertices.

\begin{figure}[htpb!]
  \centering
  \begin{tikzpicture}
    \begin{scope}
      \path (1,0) node[draw,shape=circle,fill=gray!50] (b) {$b$};
      \path (-1,0) node[draw,shape=circle,fill=gray!50] (a) {$a$};
      \path (0,-2) node[draw,shape=circle] (d) {};
      \path (1.5,-2) node[draw,shape=circle] (e) {};
      \path (-1.5,-2) node[draw,shape=circle] (c) {};

      \draw[line width=0.4mm,>=latex,->] (c) -- (a);
      \draw[line width=0.4mm,>=latex,->] (d) -- (a);
      \draw[line width=0.4mm,>=latex,->](e)--(a) ;
      \draw[line width=0.4mm,>=latex,->](c) -- (b);
      \draw[line width=0.4mm,>=latex,->](d) -- (b);
      \draw[line width=0.4mm,>=latex,->](e) -- (b);
      
      \path (0,-3) node {(a) $a$ and $b$ are open twins.};
      
    \end{scope}

    \begin{scope}[xshift=5cm]
      \path (1,0) node[draw,shape=circle,fill=gray!50] (b) {$b$};
      \path (-1,0) node[draw,shape=circle,fill=gray!50] (a) {$a$};
      \path (0,-2) node[draw,shape=circle] (d) {};
      \path (1.5,-2) node[draw,shape=circle] (e) {};
      \path (-1.5,-2) node[draw,shape=circle] (c) {};
            
      \draw[line width=0.4mm,>=latex,->] (a) -- (b);
      \draw[line width=0.4mm,>=latex,->] (c) -- (a);
      \draw[line width=0.4mm,>=latex,->] (d) -- (a);
      \draw[line width=0.4mm,>=latex,->](e)--(a) ;
      \draw[line width=0.4mm,>=latex,->](c) -- (b);
      \draw[line width=0.4mm,>=latex,->](d) -- (b);
      \draw[line width=0.4mm,>=latex,->](e) -- (b);
      \path (0,-3) node {(b) $a$ and $b$ are quasi-twins.};    
    \end{scope}
    
    \begin{scope}[xshift=10cm]
      \path (1,0) node[draw,shape=circle,fill=gray!50] (b) {$b$};
      \path (-1,0) node[draw,shape=circle,fill=gray!50] (a) {$a$};
      \path (0,-2) node[draw,shape=circle] (d) {};
      \path (1.5,-2) node[draw,shape=circle] (e) {};
      \path (-1.5,-2) node[draw,shape=circle] (c) {};    
      
      \draw[line width=0.4mm,>=latex,->] (a) to[bend right=10] (b);
      \draw[line width=0.4mm,>=latex,->] (c) -- (a);
      \draw[line width=0.4mm,>=latex,->] (d) -- (a);
      \draw[line width=0.4mm,>=latex,->] (e) -- (a) ;
      \draw[line width=0.4mm,>=latex,->] (b) to[bend right=10] (a);
      \draw[line width=0.4mm,>=latex,->] (c) -- (b);
      \draw[line width=0.4mm,>=latex,->] (d) -- (b);
      \draw[line width=0.4mm,>=latex,->] (e) -- (b);
      
      \path (0,-3) node {(c) $a$ and $b$ are closed twins.};
    \end{scope}
  \end{tikzpicture}
  \caption{Examples of twin vertices.\label{fig:twins}}
  \end{figure}
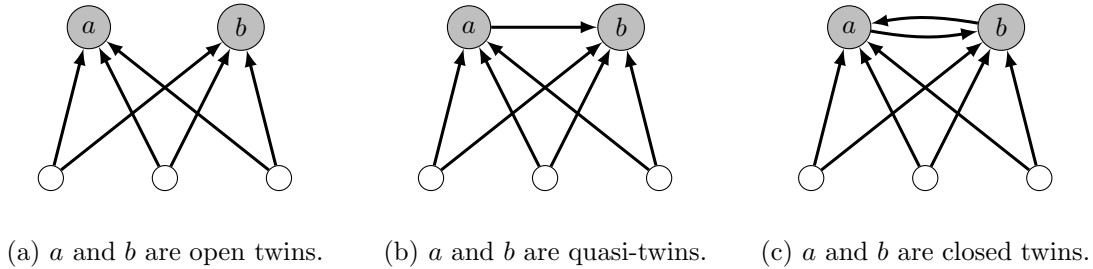

Perhaps the most classic result in domination of undirected graphs is the theorem of Ore~\cite{ore} which states that any undirected graph without isolated vertices has a dominating set of size at most half the order. Such a theorem does not hold for the location-domination number of undirected graphs, for example complete graphs and stars of order $n$ have location-domination number $n-1$, see~\cite{slater}. Nevertheless, Garijo, Gonz\'alez and M\'arquez have conjectured in~\cite{Garijo} that in the absence of twins, the upper bound of Ore's theorem also holds for the location-domination number of undirected graphs; they also proved that an upper bound of roughly two thirds the order holds in this context (see~\cite{linegraph,henning,heia} for further developments on this matter).

Therefore, it is natural to ask whether similar bounds exist in the case of locating-dominating sets of digraphs. However, Ore's theorem does not hold for digraphs. Indeed, not only every isolated vertex but also every source of a digraph $G$ belongs to any dominating set of $G$. For example, a star of order $n$ with $n-1$ sources has domination number $n-1$ (it is easy to see that this is the only digraph of order $n$ with no isolated vertices and domination number $n-1$, see~\cite{Shahrzad}). Lee showed (as part of a more general result) in~\cite{lee} that every source-free digraph $G$ of order $n$ has domination number at most $\lceil\frac{2n}{3}\rceil$. Moreover, the digraph of order $n=3k$ consisting of $k$ vertex-disjoint directed triangles is source-free and has domination number $\frac{2n}{3}$, so this bound is tight when $n=0\bmod 3$. Better bounds have been obtained for specific classes: every tournament of order $n$ has domination number at most $\lceil\log_2 n\rceil$~\cite{meggidovishkin} (the order is tight for random tournaments), and every strongly connected digraph of order $n$ has domination number at most $\lceil n/2\rceil$~\cite{leethesis} (this is tight for directed cycles).

What happens for the location-domination number? A first question of interest is to determine which digraphs have largest possible location-domination number. We address this question in Section~\ref{sec:charact-n-1}, where we describe the class of graphs of order $n$ with location-domination number $n-1$. Of course it includes stars with $n-1$ sources (that have domination number $n-1$), bidirected complete graphs and bidirected stars (that correspond to undirected graphs with location-domination number $n-1$), but as we will see, there are many more examples.

In Section~\ref{sec:general-method}, we devise a general technique to obtain small locating-dominating sets in twin-free digraphs. This technique is a refinement of those used in~\cite{heia,Garijo,Shahrzad}. We use this technique in Section~\ref{sec:generalbound} to show that every source-free and twin-free digraph of order $n$ has a locating-dominating set of size at most $\frac{4n}{5}$. This bound is improved to $\frac{3n}{4}$ if moreover the digraph has no quasi-twins. By adding one to these bounds, they also hold even in the presence of sources. In Section~\ref{sec:prelim}, we describe strongly connected twin-free and quasi-twin-free digraphs of order $n$ with location-domination number $\frac{2(n-2)}{3}$.

We then show in Section~\ref{sec:tournaments} that any tournament of order $n$ has a locating-dominating set of size at most $\lceil\frac{n}{2}\rceil$ (this is tight for transitive tournaments and other examples).

In Section~\ref{sec:acyclic}, we show that the same bound holds for twin-free acyclic digraphs (this is also tight, for directed paths).

We address some preliminary considerations in Section~\ref{sec:prelim} and conclude the paper in Section~\ref{sec:conclu}.

\section{Preliminaries}\label{sec:prelim}

We start with some useful propositions.

The following propositon generalizes a well-known fact about locating-dominating sets of undirected graphs~\cite{slater}.

\begin{proposition}\label{ldtwin-digraphs}
Let $G$ be a digraph with a set $S$ of pairwise twins (open or closed) or quasi-twins. There are at least $|S|-1$ vertices of $S$ in any locating-dominating set of $G$.
\end{proposition}
\begin{proof}
By contradiction, let $D$ be a locating-dominating set that does not contain two mutual twins or quasi-twins $x$ and $y$. The vertices $x$ and $y$ have the same in-neighbourhood in $V(G)\setminus\{x,y\}$, and also in $D$. This is a contradiction.
\end{proof}

However, we note that, unlike twins, there cannot exist a set of three pairwise quasi-twins.

\begin{proposition}\label{quasitwins}
Let $x,y,z$ be three vertices. If $x,y$ are quasi-twins and $y,z$ are quasi-twins, then $x,z$ cannot be quasi-twins.
\end{proposition}
\begin{proof}
  Suppose without loss of generality that we have the arc from $x$ to $y$. There are two cases. If we have the arc from $z$ to $y$, then we must have both arcs from $x$ to $z$ and from $z$ to $x$, so $x$ and $z$ cannot be quasi-twins. If we have the arc from $y$ to $z$, then we also have the arc from $x$ to $z$. But then $y$ is an in-neighbour of $z$ and not of $x$, so $x$ and $z$ cannot be quasi-twins.
\end{proof}

We now present a family of twin-free digraphs $G_k$ of order $n$ that have location-domination number almost $\frac{2n}{3}$. The graph $G_3$ is drawn in Figure \ref{figureLD=2n/3}.

\begin{figure}[htpb!]
\begin{center}
\begin{tikzpicture}
\path (0,-2) node[draw,shape=circle] (s) {$s$};
\path (0,4) node[draw,shape=circle,fill=gray!70] (t) {$t$};

\path (1,1) node[draw,shape=circle,fill=gray!70] (b) {};
\path (-1,1) node[draw,shape=circle] (c) {};
\path (0,2) node[draw,shape=circle,fill=gray!70] (d) {};

\path (2,1) node[draw,shape=circle,fill=gray!70] (e) {};
\path (3,2) node[draw,shape=circle,fill=gray!70] (f) {};
\path (4,1) node[draw,shape=circle] (g) {};

\path (-4,1) node[draw,shape=circle,fill=gray!70] (h) {};
\path (-3,2) node[draw,shape=circle] (i) {};
\path (-2,1) node[draw,shape=circle] (j) {};

\draw[line width=0.4mm,>=latex,->] (s) .. controls +(6,0.5) and +(6,-0.5) .. (t);

\draw[line width=0.4mm,>=latex,->] (c) -- (b);
\draw[line width=0.4mm,>=latex,->] (b) -- (d);
\draw[line width=0.4mm,>=latex,->](d)--(c) ;

\draw[line width=0.4mm,>=latex,->](e) -- (g);
\draw[line width=0.4mm,>=latex,->](g) -- (f);
\draw[line width=0.4mm,>=latex,->](f) -- (e);

\draw[line width=0.4mm,>=latex,->](h) -- (j);
\draw[line width=0.4mm,>=latex,->](j) -- (i);
\draw[line width=0.4mm,>=latex,->](i) -- (h);

\draw[line width=0.4mm,>=latex,->](b) -- (s);
\draw[line width=0.4mm,>=latex,->](c) -- (s);
\draw[line width=0.4mm,>=latex,->](d) -- (s);
\draw[line width=0.4mm,>=latex,->](e) -- (s);
\path[line width=0.4mm,>=latex,->,draw](f) to[out=-110,in=45] (s);
\draw[line width=0.4mm,>=latex,->](g) -- (s);
\draw[line width=0.4mm,>=latex,->](h) -- (s);
\path[line width=0.4mm,>=latex,->,draw](i) to[out=-80,in=135] (s);
\draw[line width=0.4mm,>=latex,->](j) -- (s);

\draw[line width=0.4mm,>=latex,->](t) -- (b) ;
\draw[line width=0.4mm,>=latex,->](t) -- (c) ;
\draw[line width=0.4mm,>=latex,->](t) -- (d);
\draw[line width=0.4mm,>=latex,->](t) -- (e) ;
\draw[line width=0.4mm,>=latex,->](t) -- (f) ;
\path[line width=0.4mm,>=latex,->,draw](t) to[out=-20,in=110] (g) ;
\path[line width=0.4mm,>=latex,->,draw](t) to[out=-160,in=70] (h) ;
\draw[line width=0.4mm,>=latex,->](t) -- (i) ;
\draw[line width=0.4mm,>=latex,->](t) -- (j) ;

\end{tikzpicture}
\end{center}
\caption{A strongly connected twin-free and quasi-twin-free digraph of order $n$ with location-domination number $\frac{2(n-2)}{3}$, with a locating-dominating set in grey.}
\label{figureLD=2n/3}
\end{figure}
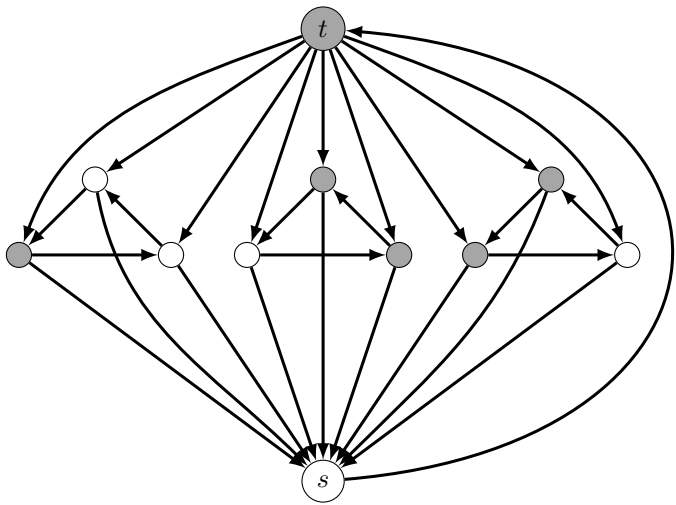

\begin{proposition}\label{prop:exampleLD=2n/3}
  Let $G_k$ be the strongly connected (hence, source-free) twin-free and quasi-twin-free digraph of order $n=3k+2$ obtained from $k$ vertex-disjoint directed triangles by adding a new vertex $s$ that is an out-neighbour of all vertices of each triangle, a vertex $t$ that is an in-neighbour of all vertices of each triangle, and an arc from $s$ to $t$. Then, we have $\LD(G_k)=\frac{2(n-2)}{3}$.
\end{proposition}
\begin{proof}
To see that $\LD(G_k)\leq\frac{2(n-2)}{3}$, consider the following locating-dominating set: take $t$, one vertex of some directed triangle, and two vertices of all the other directed triangles (see Figure~\ref{figureLD=2n/3} for an example).

To see that $\LD(G_k)\geq\frac{2(n-2)}{3}$, consider a locating-dominating set $D$ of $G_k$. First, each of the original directed triangle contains a vertex of $D$, because otherwise the three vertices in this triangle are not located. Second, there is at most one directed triangle that contains only one vertex of $D$, because otherwise in both triangles there is a vertex only dominated by $t$, and they are not located. Finally, if some triangle contains only one vertex of $D$, then $t\in D$ because otherwise some vertex of the triangle is not dominated. So, in total there are two vertices of $D$ for each directed triangle. Since there are $\frac{n-2}{3}$ original directed triangles, the proof is finished.
\end{proof}

\section{Characterizing digraphs with location-domination number $n-1$}\label{sec:charact-n-1}

In this section, we characterize those digraphs with maximum possible location-domination and location numbers.

For every digraph $G$ of order $n$, any set of size $n-1$ is a locating set, thus $\SEP(G)\leq n-1$. This is not true for locating-dominating sets: consider a digraph with no arcs. However, this is the only such example: if $G$ has some arc, then it has a locating-dominating set of size at most $n-1$ (consider the set obtained from $V(G)$ by removing the endpoint of one arbitrary arc).

We say that a vertex is \emph{universal} if it is an in-neighbour of all other vertices. We first characterize those digraphs of order $n$ with $\SEP(G)=n-1$ (of course, unless it has no arc, such a digraph also satisfies $\LD(G)=n-1$).

\begin{proposition}\label{prop:sep=n-1}
  Let $G$ be a digraph of order $n$. We have $\SEP(G)=n-1$ if and only if every vertex is either universal, or a sink.
\end{proposition}
\begin{proof}
  Recall that $\SEP(G)\leq n-1$ always holds. If some vertex $x$ has an out-neighbour $y$ and there exists another vertex $z$ that is not an out-neighbour of $x$, then $V(G)\setminus\{y,z\}$ is a locating set of $G$ of size $n-2$. Thus, if $\SEP(G)=n-1$, then every vertex in $G$ is either universal or a sink.

  On the other hand, suppose that every vertex of $G$ is either universal or is a sink. Let $U$ be the set of universal vertices, and $S$, the set of sinks. Every vertex $v$ is dominated by the set $U\cup\{v\}$. Thus, if two distinct vertices are not in a locating set $L$, they are both dominated exactly by the vertices in $L\cap U$, a contradiction. Thus we have $|L|\geq n-1$.
\end{proof}

We call a digraph $G$ a \emph{directed star} if it has a special vertex that belongs to all the arcs, and there are no isolated vertices (see Figure~\ref{fig:Dstar}). In other words, the underlying undirected graph of $G$ is a star.

\begin{figure}[!htpb]
  \centering
\begin{tikzpicture}
\path (0,0) node[draw,shape=circle] (a) {$x$};
\path (2,0.5) node[draw,shape=circle,fill=gray] (b) {};
\path (0,2) node[draw,shape=circle,fill=gray] (c) {};
\path (-2,0.5) node[draw,shape=circle,fill=gray] (d) {};
\path (-1.25,-1.75) node[draw,shape=circle,fill=gray] (e) {};
\path (1.25,-1.75) node[draw,shape=circle,fill=gray] (f) {};

\draw[->,>=latex,line width=0.4mm] (b) -- (a);
\draw[->,>=latex,line width=0.4mm] (a) -- (c);
\draw[->,>=latex,line width=0.4mm] (a) -- (d);
\draw[->,>=latex,line width=0.4mm] (e) to[bend right=20] (a);
\draw[->,>=latex,line width=0.4mm] (a) to[bend right=20] (e);
\draw[->,>=latex,line width=0.4mm] (f) to[bend right=20] (a);
\draw[->,>=latex,line width=0.4mm] (a) to[bend right=20] (f);
\end{tikzpicture}
\caption{A directed star, with a minimum locating-dominating set in grey.}\label{fig:Dstar}
\end{figure}
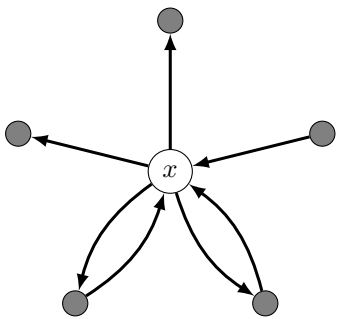

Directed stars form another family of connected digraphs with large location-domination number.

\begin{proposition}\label{prop:stars}
For any directed star $G$ of order $n\geq 2$, we have $\LD(G)=n-1$.
\end{proposition}
\begin{proof}
Let $x$ be the centre of $G$ (the vertex belonging to all arcs). Since $n\geq 2$, there is at least one arc in $G$, and so $\LD(G)\leq n-1$. Let $D$ be a locating-dominating set of $G$. Clearly, every source of $G$ belongs to $D$. If two neighbours of $x$ do not belong to $D$, then they are not located. Thus, at most one neighbour of $x$ is missing from $D$. In the case where exactly one such vertex is not in $D$, in order to have this vertex dominated, $x$ must belong to $D$. This shows that $\LD(G)\geq n-1$.
\end{proof}

 The only connected undirected graphs of order $n$ with location-domination number $n-1$ are stars and complete graphs~\cite{slater} (seen as digraphs, they correspond to bidirected stars and bidirected complete graphs). As we already mentioned, there are more digraph examples. We now characterize all of them.

\begin{theorem}\label{thm:LD=n-1}
  For a connected digraph $G$ of order $n\geq 2$, we have $\LD(G)=n-1$ if and only if at least one of the following conditions holds:
  \begin{itemize}
  \item[(a)] $n=3$;
  \item[(b)] $G$ is a directed star;
  \item[(c)] $V(G)$ can be partitoned into three (possibly empty) sets $S_1$, $C$ and $S_2$, such that $S_1$ and $S_2$ are independent sets, $C$ is a bidirected clique, and the remaining arcs in $G$ are all the possible arcs from $S_1$ to $C\cup S_2$ and those from $C$ to $S_2$.
  \end{itemize}
\end{theorem}
\begin{proof}
 Assume first that $n=2$ (then $G$ is a directed star) or $n=3$. If two vertices of $G$ are not in some locating-dominating set of $G$, then either these two are not located, or one of them is not dominated: a contradiction. Thus, necessarily $\LD(G)=n-1$ and we may assume in the remainder that $n\geq 4$.

By Proposition~\ref{prop:stars}, (b) implies that $\LD(G)=n-1$. Next, assume that (c) holds. Then, $G$ is obtained by adding the set $S_1$ to a digraph where all vertices are either universal (those in $C$) or sinks (those in $S_2$). Since all vertices in $S_1$ are sources, they all belong to any locating-dominating set of $G$. Moreover, every vertex in $S_2\cup C$ is dominated by all vertices in $S_1$, that is, we need a locating set of $G[C\cup S_2]$ in every locating-dominating set of $G$. By Proposition~\ref{prop:sep=n-1}, such a set has size exactly $|C|+|S_2|-1$, thus $\LD(G)=n-1$.

We must now prove the converse: assume that $\LD(G)=n-1$ (and $n\geq 4$). We must prove that (b) or (c) holds.

\textit{Case 1.} Suppose first that $G$ contains some sources, and let $S$ be the set of these sources. If there is a vertex $s$ in $S$ and two vertices $x$ and $y$ of $N^+(S)$ (where $N^+(S)$ denotes the union of the out-neighbourhoods of all vertices in $S$) that are located by $s$, then $V(G)\setminus \{x,y\}$ is a locating-dominating set of $G$ of size $n-2$, a contradiction. Thus all the vertices of $S$ have the same neighbourhood and $G$ contains all the arcs between $S$ and $N^+(S)$.

Consider now the subdigraph $G'$ of $G$ induced by $N^+(S)$. If it has a locating set $L'$ of size $|V(G')|-2$ (assume that the two vertices $x$ and $y$ are those not in $L'$), then the set $V(G)\setminus\{x,y\}$ would be a locating-dominating set of $G$, a contradiction. Thus, we have $\SEP(G')=|V(G')|-1$, and by Proposition~\ref{prop:sep=n-1} the vertices of $G'$ can be partitioned into sinks of $G'$ (set $S'$) and universal vertices of $G'$ (set $U'$). We let $R=V(G)\setminus (S\cup S'\cup U')$. If $R$ is empty, we are done, because then $G$ satisfies the condition (c), with $S_1=S$, $C=U'$ and $S_2=S'$. Thus, we assume that $R$ is nonempty. If $V(G)\setminus S$ contains an arc from vertex $a$ to vertex $b\in R$ such that $(S'\cup U')\setminus\{a\}$ is nonempty, then we could construct a locating-dominating set of $G$ of size $n-2$ from $V(G)$ by removing $b$ and any vertex of $(U'\cup S')\setminus\{a\}$. Thus, $R$ must be an independent set, and if there is an arc from $S'\cup U'$ to $R$, then $|S'\cup U'|=1$. But since $R$ contains no sources, there is necessarily an arc from $S'\cup U'$ to $R$, and so $|S'\cup U'|=1$. But then, $G$ is a directed star and (b) holds, so we are done.

\textit{Case 2.} Now, we assume that $G$ has no sources. If every vertex is either universal or a sink, $G$ satisfies (c) (with $S_1=\emptyset$) and we are done. Thus, we may assume that there exists a vertex $x$ with an out-neighbour $y$, and a third vertex $z$ that is not an out-neighbour of $x$. By assumption, $V(G)\setminus\{y,z\}$ is not a locating-dominating set; but since $x$ locates $y$ and $z$, it is a locating set, and it certainly dominates $y$. Thus, $V(G)\setminus\{y,z\}$ does not dominate $z$, that is, $y$ is the only in-neighbour of $z$ (recall that $z$ has an in-neighbour since $G$ has no sources). Let $R=V(G)\setminus\{x,y,z\}$. If $x$ had an out-neighbour $t$ in $R$, then $V(G)\setminus\{t,z\}$ would be locating-dominating, a contradiction. Similarly, if $x$ had an in-neighbour in $R$, then $V(G)\setminus\{x,z\}$ would be locating-dominating, a contradiction. If there is an arc from a vertex $u$ to a vertex $v$ inside $R$, then $V(G)\setminus\{v,z\}$ is a locating-dominating set of $G$, a contradiction. Thus $R$ is an independent set. Now, if $z$ has no neighbour in $R$, $G$ is a directed star with centre $y$, and we are done. Thus, $z$ must have an out-neighbour $p$ in $R$. But now, $V(G)\setminus\{p,y\}$ is locating-dominating, a contradiction. This completes the proof.
\end{proof}

\section{A general method to obtain locating-dominating sets of twin-free digraphs}\label{sec:general-method}

Note that all graphs described in Section~\ref{sec:charact-n-1} of order $n\geq 4$ have twins. What happens for twin-free digraphs?

In this section, we propose a general method to obtain locating-dominating sets of twin-free digraphs, based on special dominating sets. This method was used in~\cite{heia} for the case of undirected graphs (a similar argument was also used in~\cite{Garijo}). It was adapted to digraphs in~\cite{Shahrzad} for quasi-twin-free digraphs, and here we extend it to all twin-free digraphs.

First we start with some definitions.
\begin{definition}
  Let $S$ be a set of vertices of a digraph $G$. The \emph{$S$-partition} $\mathcal{P}_S$ of $V(G)\setminus S$ is the partition of $V(G)\setminus S$ where two vertices are in the same part if and only if they have the same set of in-neighbours in $S$.

Given a vertex $v\in S$, an \emph{$S$-external private neighbour} of $v$ is a vertex outside $S$ that is an out-neighbour of $v$ but of no other vertex of $S$ in $G$.
\end{definition}

\begin{theorem}\label{thm:extended-general-method}
Suppose that $G$ is a twin-free digraph of order $n$. Let $S$ be a dominating set of $G$ such that the $S$-partition of $V(G)\setminus S$ contains at least $x\cdot|S|$ parts (with $0< x\leq 1$). Then, $\LD(G) \leq\frac{2x+1}{3x+1}n$. Moreover, if $G$ is also quasi-twin-free, then $\LD(G) \leq\frac{x+1}{2x+1}n$.
\end{theorem}
\begin{proof} 
Let $\mathcal P_S=P_1 \cup\ldots\cup P_{n_1}\cup Q_1\cup\ldots \cup Q_{n_2}$ be the $S$-partition of $V(G)\setminus S$, where $P_1,\ldots, P_{n_1}$ are the parts of size~$1$ and $Q_1,\ldots, Q_{n_2}$ are the parts of size at least~$2$.

We assume that $S$ is maximal with the property that $\mathcal P_S$ has at least $x\cdot |S|$ parts (this is ensured by adding vertices to $S$ while this property holds).

Now, we let $D_1 = S\cup \bigcup_i P_i$. We have the following property of $D_1$.

\begin{claim}\label{claim:D1}
Two vertices in $V(G)\setminus D_1$ are located by $D_1$, unless they form a pair of quasi-twins.
\end{claim}
\claimproof Clearly, if two vertices are in different parts of $\mathcal P_S$, they are located by some vertices in $S$. Thus, by contradiction, let $q_1$ and $q_2$ be two vertices of $V(G) \setminus D_1$ belonging to some part $Q_{i_0}$ of $\mathcal P_S$ that are not quasi-twins but are not located by $D_1$. Since $G$ is twin-free, there is a vertex $q_3$ in $V(G)\setminus S$ that can locate $q_1$ and $q_2$: without loss of generality $q_3$ is an in-neighbour of $q_1$ but not $q_2$. By our assumption $q_3 \notin D_1$. Now, consider $S'= S\cup \lbrace q_3\rbrace$ and the corresponding $S'$-partition $\mathcal P_{S'}$ of $V(G)\setminus S'$ (with $n'_1$ and $n'_2$ defined like before). Since $q_3\in\bigcup_i Q_i$, we have $n'_1+n'_2 \geq n_1+n_2+1$ (because in $\mathcal{P}_{S'}$, $Q_{i_0}$ has been split into two parts). So $n'_1 + n'_2 \geq x |S|+1\geq x(|S|+1)=x|S'|$ (because $x\leq 1$). This contradicts the choice of $S$, which we assumed to be maximal with this property.~\smallqed

\medskip

Note that $|D_1|=|S|+n_1$. Since $D_1$ is a dominating set, Claim~\ref{claim:D1} shows that in the absence of quasi-twins, $D_1$ is locating-dominating. Next, we address the case of quasi-twins. We first prove the following fact.

\begin{claim}\label{claim:quasi-twins}
Any two pairs of quasi-twins in $V(G)\setminus D_1$ are disjoint.
\end{claim}
\claimproof Note that two quasi-twins $x$ and $y$ in $V(G)\setminus S$ must belong to the same part of $\mathcal P_S$, since they have the same in-neighbours in $S$. Let $q_1$, $q_2$, $q_3$, $q_4$ be four vertices in $V(G)\setminus D_1$ such that $\{q_1,q_2\}$ and $\{q_3,q_4\}$ are two distinct pairs of quasi-twins (with $q_1$ and $q_3$ being in-neighbours of $q_2$ and $q_4$, respectively). Assume by contradiction that the two pairs are not disjoint. Then, all the vertices in the two pairs belong to the same part of $\mathcal P_S$. If $q_1=q_3$, then $q_2$ and $q_4$ must be twins (because they cannot be quasi-twins by Proposition~\ref{quasitwins}), a contradiction. Similarly, if $q_2=q_4$, then, because $q_1$ and $q_4$, and $q_3$ and $q_4$ are quasi-twins, there must be an arc from $q_1$ to $q_3$ and an arc from $q_3$ to $q_1$. Thus, $q_1$ and $q_3$ are twins. Hence, we may assume without loss of generality that $q_2=q_3$. Then, we proceed similarly as in Claim~\ref{claim:D1}: the set $S'=S\cup\{q_2\}$ still satisfies the property that $n'_1 + n'_2\geq x|S'|$, contradicting the maximality of $S$. This proves the claim.~\smallqed

\medskip

Now, for each pair of quasi-twins in $V(G)\setminus D_1$, we add one of them to $D_1$. By Claims~\ref{claim:D1} and~\ref{claim:quasi-twins}, the resulting set $D_1'$ is locating-dominating and has size at most $|S|+n_1+(n-|S|-n_1)/2=(n+|S|+n_1)/2$.

 Consider now the set $D_2$, of size $n-n_1-n_2$, consisting of $V(G)$ without one (arbitrary) vertex from each part of $\mathcal P_S$. It is clear that $D_2$ is locating-dominating: all vertices of $V(G)\setminus D_2$ are located and dominated by $S$.

We now assume that $G$ has no quasi-twins: then $D_1$ and $D_2$ are two locating-dominating sets of $G$. If $|D_2|\leq\frac{x+1}{2x+1}n$, we are done. So, assume that $|D_2|> \frac{x+1}{2x+1}n$. We have $|D_2|=n-n_1-n_2$, so $n_1 + n_2 <(1-\frac{x+1}{2x+1})n=\frac{x}{2x+1}n$. Recall that $|S| \leq \frac{n_1+n_2}{x}$. Therefore,
\begin{align*}
|D_1| &=|S|+n_1 \\
&\leq |S|+n_1 +n_2\\
&\leq \frac{n_1+n_2}{x} + (n_1+ n_2)\\
&= \left(\frac{1}{x} +1\right)\left(n_1+n_2\right) \\
&<\left(\frac{1}{x} +1\right)\left(\frac{x}{2x+1}n\right)\\
&=\frac{x+1}{2x+1}n\;,
\end{align*}
as desired.

If $G$ has some quasi-twins, we use the locating-dominating sets $D_1'$ and $D_2$. Again, if $|D_2|\leq\frac{2x+1}{3x+1}n$, we are done. So, assume that $|D_2|> \frac{2x+1}{3x+1}n$. Then, $n_1 + n_2 <(1-\frac{2x+1}{3x+1})n=\frac{x}{3x+1}n$. We obtain:
\begin{align*}
|D_1'| &\leq\frac{|S|+n+n_1}{2}\\
&\leq \frac{|S|+n+n_1+n_2}{2}\\
&\leq \frac{n_1+n_2}{2x}+\frac{n}{2}+\frac{n_1+n_2}{2}\\
&< \frac{n}{6x+2}+\frac{n}{2}+\frac{x}{6x+2}n\\
&= \frac{2x+1}{3x+1}n\;,
\end{align*}
and the proof is finished.
\end{proof}

\section{Applying Theorem~\ref{thm:extended-general-method} to general twin-free digraphs}\label{sec:generalbound}

We first apply the method of Section~\ref{sec:general-method} for source-free digraphs.

\begin{proposition}\label{prop:DS-S/2-parts}
Any source-free digraph $G$ has a minimum-size dominating set $S$ with at least $|S|/2$ distinct parts in the $S$-partition of $V(G)\setminus S$.
\end{proposition}
\begin{proof}
Let $S$ be a minimum-size dominating set. We choose $S$ with a maximum number of parts in the partition $\mathcal P(S)$ of $V(G)\setminus S$.
  
Let $S_1$ be the set of vertices of $S$ that have at least one $S$-external private neighbour. Next, we choose $S_2$ as a minimum-size set of vertices of $S\setminus S_1$ that dominates all the vertices of $(V(G)\setminus S)\setminus N^+(S_1)$ (that is, $S_2$ dominates those vertices that are not in $S$ and are not an $S$-external private neighbour of any vertex in $S$). Note that $S_1\cap S_2=\emptyset$ and $S_1\cup S_2$ dominates $V(G)\setminus S$. Finally, we let $S_3=S\setminus (S_1\cup S_2)$.

Note that $|S_2|$ is at most the number of parts of $\mathcal P(S)$ containing vertices with at least two in-neighbours in $S$, since a vertex of $S$ suffices to dominate the vertices of each such part. Therefore, the number of parts in $\mathcal P(S)$ is at least $|S_1|+|S_2|$.

We will now show that we have $|S_1|+|S_2|\geq |S|/2$, which would imply the statement. Towards a contradiction, we assume that $|S_1|+|S_2|<|S|/2$, that is, $|S_3|>|S|/2$.

Then, no vertex $x$ of $S_3$ has an in-neighbour in $S$, for otherwise $S\setminus\{x\}$ would be a dominating set, contradicting the minimality of $S$. Since $G$ is source-free, $x$ has an in-neighbour in $V(G)\setminus S$. Let $f(x)$ be one arbitrary in-neighbour of $x$, and let $S_4=\{f(x)~|~x\in S_3\}$. The function $f$ is necessarily injective: if we had $f(x_1)=f(x_2)=y$ for two distinct vertices $x_1$ and $x_2$ of $S_3$, then the set $S\setminus\{x_1,x_2\}\cup\{y\}$ would be a smaller dominating set than $S$. Thus $|S_4|=|S_3|$.

Now, we let $S'=S_1\cup S_2\cup S_4$. Clearly, $S'$ is a dominating set of $G$ of size $|S|$. Moreover, every vertex $x$ of $S_3$ is an $S'$-external private neighbour of $f(x)$. 
Thus, the partition $\mathcal P(S')$ of $V(G)\setminus S'$ has at least $|S_3|$ parts, which is more than $|S|/2=|S'|/2$. This contradicts the choice of $S$.

Thus, as claimed, we have $|S_1|+|S_2|\geq |S|/2$, which concludes the proof.
\end{proof}

We remark that the bound of Proposition~\ref{prop:DS-S/2-parts} is tight by considering any digraph consisting only of vertex-disjoint directed triangles.

We obtain the immediate consequence of Proposition~\ref{prop:DS-S/2-parts} and Theorem~\ref{thm:extended-general-method}.

\begin{corollary}\label{coro:source-free-twin-free}
  For any source-free and twin-free digraph $G$ of order $n$, we have $\LD(G)\leq 4n/5$. If moreover $G$ is quasi-twin-free, then $\LD(G)\leq 3n/4$.
\end{corollary}

One can extend this result to twin-free digraphs with sources (note that there can only be one source, since multiple sources would be mutual open twins).

\begin{corollary}\label{coro:twin-free}
  Let $G$ be a twin-free digraph of order $n\geq 3$. Then $\LD(G)\leq 4n/5+1$. If moreover $G$ is quasi-twin-free, then $\LD(G)\leq 3n/4+1$.
\end{corollary}
\begin{proof}
  Let $s$ be the unique source of $G$.
  Let $\mathcal I$ be the collection of all subsets $N^-(x)$ for $x\in V(G)$.
  The set $\mathcal I$ has order $n$ and thus there is a non-empty subset $X$ of $V(G)\setminus \{s\}$ that is not in $\mathcal I$, and thus $X$ is not the open in-neighbourhood of any vertex in $G$.
    Let $G'$ be the graph obtained from $G$ in which we add all the arcs between $X$ and $s$.
    The graph $G'$ is now source-free and twin-free. By Corollary~\ref{coro:source-free-twin-free}, there is a locating-dominating set $D'$ of $G'$ of size at most $4n/5$. Then $D=D'\cup \{s\}$ is a locating-dominating set of $G$. Indeed, all the vertices are dominated and two vertices not in $D$ are still located by $D'$. Thus $\LD(G)\leq 4n/5+1$.

    For the second part, we do the same reasonning but with $\mathcal I$ containing all the subsets $N^-[x]\setminus \{s\}$ and $N^-(x)\setminus \{s\}$, for $x\in V(G)$. There are at most $2n$ such sets and thus again there is a nonempty set $X$ such that adding the arcs between $x$ and $s$ does not create twins or quasi-twins in $G$. The end of the proof is the same as in the first part.
    \end{proof}

\section{Tournaments}\label{sec:tournaments}

It is clear that there are no twins in tournaments. A weaker version of the method of Section~\ref{sec:general-method} was applied in~\cite{Shahrzad}, yielding the bound $\LD(T)\leq 5n/6$ for any tournament $T$. This was proved by showing that any tournament $T$ has a dominating set $S$ with at least $|S|$ parts in the $S$-partition of $V(T)\setminus S$.

Note that our strengthened version of this method yields the better bound $\LD(T)\leq 4n/5+1$, via Corollary~\ref{coro:twin-free}.

Nevertheless, using a different technique, we next prove a much stronger bound, which, moreover, is tight. We first prove this bound in transitive tournaments (for which it is actually the exact value).

\begin{proposition}\label{transitivetourn}
For a transitive tournament $T$ of order $n$, we have $\LD(T)=\lceil\frac{n}{2}\rceil$ and $\SEP(T)=\lfloor \frac{n}{2} \rfloor$.
\end{proposition}
\begin{proof}
Let $V(T)=\{v_1,\ldots,v_n\}$ where $v_i$ has each $v_j$ with $j>i$ as its out-neighbour. 
To see that $\gamma(T)\leq\lceil\frac{n}{2}\rceil$, consider the dominating set containing all $v_i$'s with $i$ odd. Two vertices $v_i$ and $v_j$ ($i<j$) not in $D$ are located by $v_{j-1}$.
Now, let $D$ be a locating-dominating set of $T$. We need $v_1$ in $D$, otherwise it is not dominated. Then, every two vertices $v_i$, $v_{i+1}$ are quasi-twins, thus by Proposition \ref{ldtwin-digraphs}, one of them needs to be in $D$. Consider the sets $\{v_{i},v_{i+1}\}$ with $i$ even and $i<n$: they are disjoint, and each contains one vertex of $D$. There are $\lfloor\frac{n-1}{2}\rfloor$ such sets, so we have 
\begin{align*}
|D|&\geq\left\lfloor\frac{n-1}{2}\right\rfloor+1=\left\lceil\frac{n}{2}\right\rceil.
\end{align*}

For locating sets, if $n$ is odd, consider the set $L$ that contains all $v_i$'s with $i$ even. Then $L$ is a locating set of size $(n-1)/2$. Since $\SEP(T)\geq \LD(T)-1$, this set is optimal.
If $n$ is even, at least one vertex in any set  $\{v_{i},v_{i+1}\}$ with  $i$ odd must be in the locating set. This gives at least $n/2$ vertices and thus $\SEP(T)=\LD(T)$.
\end{proof}

We now extend the upper bound to any tournament.

\begin{theorem}\label{thm:tournaments-n/2}
 For any tournament $T$ of order $n$, we have $\LD(T)\leq \lceil n/2 \rceil$ and $\SEP(T)\leq \lfloor n/2 \rfloor$.
\end{theorem}
\begin{proof}
  We prove the result by induction. By Proposition \ref{transitivetourn}, this is true for any transitive tournament, and in particular if $n\leq 2$. Let $n\geq 3$ and  assume that this is true for any tournament of order $k<n$.

  Let $T$ be a tournamement of order $n$ that is not transitive. We first find a locating set of size $\lceil n/2 \rceil$.

  Let $x$ be any vertex. Let $V_0=N^-(x)$ and $V_x=N^+(x)$ be the $\{x\}$-partition of $V(T)\setminus\{x\}$. Let $n_0$ and $n_x$ be the sizes of $V_0$ and $V_x$, respectively. 
  Let $S_0$ and $S_x$ be two optimal locating sets of the tournaments induced by $V_0$ and $V_x$, respectively.
  By induction, $S_0$ and $S_x$ have size at most $\lfloor n_0/2 \rfloor$ and $\lfloor n_x/2 \rfloor$.
  Consider the set $S=S_0\cup S_x \cup \{x\}$. It is a locating set of $T$ since any pair $u, v$ with $u \in V_0$ and $v \in V_x$ is located by $x$. Its size is at most $\lfloor n_0/2 \rfloor+\lfloor n_x/2 \rfloor+1$ which is equal to $\lfloor n/2 \rfloor$ if $n_0$ or $n_x$ is odd. In this case we are done, so we can assume that both $n_0$ and $n_x$ are even and thus $n$ is odd. Since we chose an arbitrary vertex $x$, one can also assume that all the vertices have even out-degree and in-degree (if not, we are done).

  Consider now two arbitrary vertices $x$ and $y$ with an arc from $x$ to $y$. Let $V_x$, $V_y$, $V_{xy}$, $V_0$ be the $\{x,y\}$-partition of $V\setminus \{x,y\}$ ($V_x$ contains the vertices that have $x$ but not $y$ in their in-neighbourhood and the other notations follow the same logic). As before, if one takes a minimum locating set in each part of the partition and add $x$ and $y$, one obtains a locating set of $T$. If there are three odd-sized sets among the four sets, one obtains, using the induction hypothesis, a locating set of size at most $\lfloor n/2 \rfloor$.
  Note that if $V_{xy}$ has odd size, then so does $V_y$ since $y$ has even out-degree and all its out-neighbours are in $V_y$ or $V_{xy}$. Since the total number of vertices is odd, there is another odd-sized set among $V_0$ and $V_x$ (which must actually be $V_0$). In total, there are three odd-sized sets among the sets of the partition and thus the locating set has size at most $\lfloor n/2\rfloor$.
  Thus, one can assume that for any pair of vertices of $T$, the number of common out-neighbours is even.

  We now consider three vertices $x$, $y$ and $z$ that induce a directed triangle $x\to y\to z\to x$ (this exists since $T$ is not transitive). Again, we consider the $\{x,y,z\}$-partition of $V(T)\setminus \{x,y,z\}$, we consider a minimum locating set of each part and we add $x$ and $y$ and $z$ to obtain a locating set. If there are four odd-sized sets in the partition, the obtained locating set has, by induction, size at most $\lfloor n/2\rfloor$.
  In fact, this is always the case: if $V_{xyz}$ (the vertices that have $x$, $y$ and $z$ as in-neighbours) is odd-sized, then $V_{xy}$ is also odd-sized since $V_{xy}\cup V_{xyz}$ is exactly the set of common out-neighbours of $x$ and $y$, which is even-sized. In the same way, $V_{yz}$ and $V_{xz}$ are odd-sized and we are done.
  If $V_{xyz}$ is even-sized, then using the same argument, $V_{xy}$, $V_{yz}$, $V_{xz}$ are also even-sized.
  But then, $V_x\cup V_{xy} \cup V_{xz} \cup V_{xyz} \cup \{y\}$ is the out-neighbourhood of $x$ and has even size. Thus $V_x$ is odd-sized, and similarly, $V_y$ and $V_z$ are also odd-sized. Since the order of $T$ is odd, $V_0$ must also be odd-sized and we are done.

  We now prove in a similar way that $\LD(T)=\lceil n/2 \rceil$. If $n$ is odd, this is clear since $\LD(T)\leq \SEP(T) +1\leq \lceil n/2 \rceil$ by the previous result. Thus we assume that $n$ is even.
  We first take an arbitrary vertex $x$ and consider the $\{x\}$-partition $V_0$,$V_x$.
  If one takes a locating-dominating set for $V_0$, a locating set for $V_x$ and adds $x$, one obtains a locating-dominating set. Since $n$ is even, exactly one set among $V_0$ and $V_x$ has odd size.
  If $|V_0|$ is even, by induction, this gives a locating-dominating set of size at most $|V_0|/2+(|V_x|-1)/2+1=n/2$ and we are done. Thus, one can assume that all the vertices have odd in-degree and even out-degree.

  Consider now two arbitrary vertices $x$ and $y$ with an arc from $x$ to $y$ in $T$ and the associated $\{x,y\}$-partition. Taking a locating-dominating set for $V_0$, locating sets for the three other parts and $x$ and $y$ gives a locating-dominating set. If two sets of the partition that are not $V_0$ are odd-sized, this gives a locating-dominating set of size $n/2$, and we are done. So, we can assume that among $V_x$ and $V_{xy}$, exactly one set is odd-sized. Indeed, $x$ has even out-degree and its out-neighbourhood is $V_x\cup V_{xy} \cup \{y\}$. Thus, if $V_y$ is odd-sized, we are done. Hence, we can assume that $V_y$ is even-sized, which implies that $V_{xy}$ is also even-sized.
  In particular, we can now suppose that all pairs of vertices have a common out-neighbourhood that has even size.

  Finally, consider an oriented triangle $x\to y \to z \to x$ and the associated $\{x,y,z\}$-partition. As before, by taking $x$, $y$, $z$, a locating-dominating set of $V_0$ and locating sets in the seven other parts, we obtain a locating-dominating set.
  Since $n$ is even, there is an odd number of odd-sized sets in the partition. If there are three odd-sized sets that are not $V_0$, then the total locating-dominating set has size at most $n/2$ (indeed, if $V_0$ is also odd-sized, then there will be a fourth odd-sized set that is not $V_0$).
  As before, if $V_{xyz}$ is odd-sized, then $V_{xy}$, $V_{xz}$, $V_{yz}$ are also odd-sized and we are done.
  If it is even-sized, since $x$, $y$ and $z$ have even out-degree, we have $V_x$, $V_y$ and $V_z$ that are odd-sized and we are also done.

  Thus, there is always a locating-dominating set of $T$ of size $\lceil n/2\rceil$.
\end{proof}

Transitive tournaments are not the only tight example. For an integer $k\geq 1$, let $T_k$ be the tournament of order $3k$ obtained from a collection $t_1,\ldots, t_k$ of vertex-disjoint directed triangles, with arcs going from all vertices of $t_i$ to all vertices of $t_j$ whenever $i<j$.

\begin{proposition}\label{prop:tourn-disj-triangles}
The tournament $T_k$ of order $n=3k$ satisfies $\LD(T_k)=\lceil n/2\rceil$.
\end{proposition}
\begin{proof}
  Let $D$ be an optimal locating-dominating set of $T_k$.
  There must be at least one vertex of $D$ in any $t_i$, otherwise the vertices of $t_i$ are not located.
  Furtheremore, there must be two vertices of $D$ in $t_1$ to dominate the three vertices of $t_1$.

  Assume that there are two consecutive triangles, $t_i$ and $t_{i+1}$, each with only one vertex in $D$. Let $d_i$ and $d_{i+1}$ be the vertices of $D$ that are in $t_i$ and $t_{i+1}$. Then, the out-neighbour of $d_i$ in $t_i$ and the in-neighbor of $d_{i+1}$ in $t_{i+1}$ are not located, a contradiction.

  Thus, there must be at least three vertices of $D$ in any two consecutive triangles, and at least two vertices of $D$ in the first triangle, which gives a total of at least $k+\lceil k/2\rceil=\lceil n/2\rceil$ vertices in $D$.
\end{proof}

\section{Twin-free acyclic digraphs}\label{sec:acyclic}

Twins generalize sources (since two sources are twins) but in a twin-free digraph we may have up to one source. This allows us to consider twin-free acyclic digraphs (they have exactly one source) in this section.

We will need the following theorem of Bondy, rephrasesd for our context.

\begin{theorem}[Bondy \cite{bondy}]\label{bondy}
Let $A$, $B$ be two disjoint sets of vertices in a digraph such that the vertices of $B$ have distinct in-neighbourhoods in $A$. Then, there is a subset $L\subseteq A$ of size at most $|B|-1$ such that the vertices of $B$ have distinct in-neighbourhoods in $L$.
\end{theorem}

Now we can prove the following. The proof technique is similar, but more complicated, as the one used to prove the same bound for the domination number of twin-free digraphs in~\cite{Shahrzad}.

\begin{theorem}\label{LDacyclic}
If $G$ is a twin-free acyclic digraph of order $n$, then $\LD(G) \leq \lceil\frac{n}{2}\rceil$.
\end{theorem}
\begin{proof}
Let $s$ be the unique source. We partition the vertex set step by step. For $i\geq 0$, the set $L_i$ contains all sources in $G_i = V(G)\setminus\bigcup_{j < i} L_j$. Since $G$ is acyclic and has only one source all the vertices are in some $L_i$. Let $m$ be the last non-empty $L_i$. See Figure \ref{figureLDacyclic} for a picture.

\begin{figure}[htpb!]
\begin{center}
\scalebox{1}{\begin{tikzpicture}

\path (0,6) node[draw,shape=circle] (s) {s};
\path (0,4) node[draw,shape=circle] (a1) {};
\path (-1,2) node[draw,shape=circle] (b2) {};
\path (0,2) node[draw,shape=circle] (b3) {};
\path (-1,0) node[draw,shape=circle] (c1) {};
\path (0,0) node[draw,shape=circle] (c2) {};
\path (1,0) node[draw,shape=circle] (c3) {};

\draw[rounded corners] (-2.5, 3.5) rectangle (2, 4.5) {};
\draw[rounded corners] (-2.5, 1.5) rectangle (2, 2.5) {};
\draw[rounded corners] (-2.5, -0.5) rectangle (2, 0.5) {};
\draw[rounded corners] (-2.5, 5.5) rectangle (2, 6.5) {};

\path (3,6) node {$L_0$};
\path (3,4) node {$L_1$};
\path (3,2) node {$L_2$};
\path (3,0) node {$L_3$};

\draw[>=latex,->,line width=0.4mm] (s) -- (a1);
\draw[>=latex,->,line width=0.4mm] (a1) -- (b3);
\draw[>=latex,->,line width=0.4mm] (a1) -- (b2);
\draw[>=latex,->,line width=0.4mm] (b2) -- (c1);
\draw[>=latex,->,line width=0.4mm] (b2) -- (c2);
\draw[>=latex,->,line width=0.4mm] (b3) -- (c2);
\draw[>=latex,->,line width=0.4mm] (b3) -- (c3);
\draw[>=latex,->,line width=0.4mm] (s) .. controls +(1,-1) and +(1,1) .. (c3);
\draw[>=latex,->,line width=0.4mm] (s) .. controls +(-1,-1) and +(-1,1) .. (b3);
\draw[>=latex,->,line width=0.4mm] (a1) .. controls +(-1,-1) and +(-1,1) .. (c1);

\end{tikzpicture}}
\end{center}
\caption{The levels in the proof of Theorem \ref{LDacyclic}.}
\label{figureLDacyclic}
\end{figure}
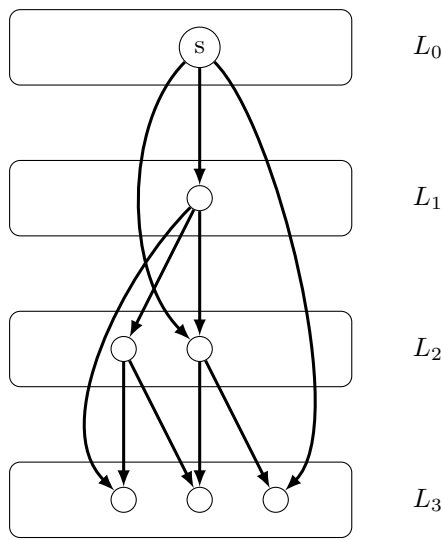

The following claims are a direct consequence of the construction.

\begin{claim}\label{father}
Let $v \in L_i$ ($i>0$). Then $v$ has an in-neighbour in $L_{i-1}$.
\end{claim}
\begin{claim}
There is no arc inside any set $L_i$.
\end{claim}
\begin{claim}
There is no arc from $L_j$ to $L_i$ for $j > i$.
\end{claim}

Now we construct a locating-dominating set $D$ of $G$.
For $i=m, \dots, 1$, we construct step by step subsets $D_i$ of $D$. We will also define $L'_i=L_i\setminus\bigcup_{j>i} D_j$, where at the beginning of the process ($i=m$) we let $L'_m=L_m$.

We will ensure that $D_i\cap L_{i-1}$ dominates $L'_i$, and $D_i$ locates the vertices of $L'_i$.


Let $P_1 ,\ldots , P_k$ be the $L_{i-1}$-partition of $L'_i$. In particular, vertices in this partition are twins with respect to $L_{i-1}$. For each $j$, let $v_j$ be any vertex of $P_j$. By  Bondy's theorem (Theorem~\ref{bondy}), there is a set $S$ of at most $k$ vertices in $L_{i-1}$ that locates and dominates $v_1 ,\ldots, v_k$.

 Again by applying Bondy's theorem, for any $1\leq j\leq k$, there is a set $S_j$ of $|P_j| -1$ vertices in $G$ that locates $P_j$. Note that the vertices of $S_j$ belong to $\bigcup_{j<i-1} L_j$. Let  $D_i=S\cup \bigcup_{j=1}^k S_j$. Then $D_i\cap L_{i-1}=S$ dominates $L'_i$ and $D_i$ locates $L'_i$. 
\begin{center}
$|D_i|\leq k + \sum_{j=1}^{k}{\left(|P_j|-1\right)} = |L'_i|.$
\end{center}

Now we prove that $D=\bigcup_{i=1}^m D_i\cup \{s\}$ is a locating-dominating set. The set $D$ is a dominating set because for any $x$ out of $D$ ($x\in L_i$), $x$ is dominated by  $D_i\cap L_{i-1}$. It is also locating because if there are two vertices $u, v$ in $L_i$, they are located by $D_i$ and if they are from different $L_i$ and $L_j$ for $j<i$, $u\in L_i$ is located by its in-neighbour in $L_{i-1} \cap D_i$.

We now bound the size of $D$. We have $|D_i|\leq |L'_i|$ at each step $i=m, \ldots, 1$ of the construction. Thus, there are at most as many vertices in $D'=\bigcup_{i=1}^m D_i$ as in $V(G)\setminus D'$ ($D_i$ only contains vertices that are in $L_j$ with $j<i$, hence $V(G)\setminus D$ is exactly the set of all vertices in the sets $L'_i$). So, if $D'=D$ (that is, $D'$ contains $s$), we  have $|D|\leq\lfloor\frac{n}{2}\rfloor$. But if $D'$ does not contain $s$, we have in fact that there are at most as many vertices in $D'$ as in $V(G)\setminus (D'\cup\{s\})$. Thus, 
$|D'|\leq\lfloor\frac{n-1}{2}\rfloor$ and 
\begin{align*}
|D|&=|D'|+1 \\
&\leq\left\lfloor\frac{n+1}{2}\right\rfloor \\
&=\left\lceil\frac{n}{2}\right\rceil,
\end{align*}
as wished.
\end{proof}

The bound of Theorem \ref{LDacyclic} is best possible by considering directed paths (the proof of the following is easy, so we omit it).

\begin{proposition}\label{prop:directedpath}
For the directed path $P_n$ of order $n$, we have $\gamma(P_n)=\gamma_L(P_n)=\left\lceil\frac{n}{2}\right\rceil$.
\end{proposition}

\section{Conclusion}\label{sec:conclu}

We conclude the paper with Table~\ref{table}, summarizing the known upper bounds on the domination and location-domination numbers for certain classes of digraphs.

\begin{table}[htpb!]
 \begin{center}
\scalebox{0.9}{\begin{tabular}{c|c|c}
~&~&\\[-0.2cm]
class of digraphs & $\gamma$ & $\LD$ \\ [-0.2cm]
~&~&\\
\hline
\hline
~&~&\\[-0.2cm]
source-free   & $\left\lceil\frac{2n}{3}\right\rceil$ \cite{lee} & $n-1$ [Prop \ref{prop:stars}]\\[-0.1cm]
&  & \\
\hline
~&~&\\[-0.2cm]
twin-free and   & $\left\lceil\frac{2n}{3}\right\rceil$ \cite{lee} & $\leq \frac{4n}{5}$ [Cor \ref{coro:source-free-twin-free}] ($\geq\frac{2(n-2)}{3}$ [Prop \ref{prop:exampleLD=2n/3}]) \\[-0.1cm]
source-free &  & \\
\hline
~&~&\\[-0.2cm]
twin-free, source-free and   & $\left\lceil\frac{2n}{3}\right\rceil$ \cite{lee} & $\leq \frac{3n}{4}$ [Cor \ref{coro:source-free-twin-free}] ($\geq\frac{2(n-2)}{3}$ [Prop \ref{prop:exampleLD=2n/3}]) \\[-0.1cm]
quasi-twin-free &  & \\
\hline
~&~&\\[-0.2cm]
tournaments     & $\left\lceil\log_2 n\right\rceil$ \cite{meggidovishkin} & $\left\lceil\frac{n}{2}\right\rceil$ [Thm~\ref{thm:tournaments-n/2}, Props~\ref{transitivetourn},\ref{prop:tourn-disj-triangles}] \\[-0.1cm]
&  & \\
\hline
~&~&\\[-0.2cm]
acyclic twin-free   & $\left\lceil\frac{n}{2}\right\rceil$ [Thm \ref{LDacyclic}, Prop~\ref{prop:directedpath}] & $\left\lceil\frac{n}{2}\right\rceil$ [Thm \ref{LDacyclic}, Props~\ref{transitivetourn},\ref{prop:directedpath}]\\[-0.1cm]
&  & \\
\hline
~&~&\\[-0.2cm]
strongly connected     & $\left\lceil\frac{n}{2}\right\rceil$ \cite{leethesis} & $n-1$ [Prop \ref{prop:stars}] \\[-0.1cm]
&  & 
\end{tabular}}
\end{center}
\caption{Summary of known upper bounds for domination and location-domination numbers of digraphs, with references. Cells with just one number correspond to tight bounds. For every other cell, we indicate the known upper bound and the size of the largest known construction.}
\label{table}
 \end{table}

The main question arising from our results is whether every twin-free digraph of order $n$ admits a locating-dominating set of size $\frac{2n}{3}$. Also, it would be interesting to determine which are the tournaments and the twin-free acyclic digraphs of order $n$ with location-domination number exactly $\lceil\frac{n}{2}\rceil$.

We also ask whether Proposition~\ref{prop:DS-S/2-parts} could be strengthened in the following sense: is it true that every source-free digraph has a minimum-size dominating set $S$ with $|S|/2$ vertices in $S$ having an $S$-external private neighbour? This would be best possible, and an analogue of a similar result from~\cite{bollobas}, which holds for undirected graphs (with $|S|/2$ replaced with $|S|$). Moreover, we do not know whether Proposition~\ref{prop:DS-S/2-parts} is tight for strongly connected digraphs.

\paragraph{Acknowledgements} We thank the anonymous referees for their detailed comments which helped us to improve the presentation of the paper.

\end{document}